\documentclass[11pt,english,12p]{amsart}
\usepackage{amsmath}
\usepackage{amssymb}
\usepackage{tikz}
\usepackage[autostyle]{csquotes}
\usepackage{mathrsfs}
\usepackage{float}
\usepackage{tikz-cd} 
\usepackage{verbatim}
\usepackage{hyperref}

\usepackage{pst-node}
\usetikzlibrary{patterns}
\usetikzlibrary{arrows,positioning,shapes,fit,calc}

\usepackage{scalerel}
\usepackage{stackengine,wasysym}

 \usepackage[all,cmtip]{xy}

\usetikzlibrary{arrows,positioning,shapes,fit,calc}

\newtheorem{thm}{Theorem}[section]
\newtheorem{lem}[thm]{Lemma}
\newtheorem{prop}[thm]{Proposition}

\newtheorem{defn}[thm]{Definition}

\newcommand{\R}{{\mathbb R}}

\newcommand{\C}{{\mathbb C}}

\newcommand{\End}{\mathrm{End}}
\newcommand{\rk}{\mathrm{rk}}

\selectfont

\title{Complex tori constructed from Cayley--Dickson algebras}

\author[I. Grzegorczyk and R. Su\'arez]{Ivona Grzegorczyk and Ricardo Su\'arez}
\address{Dept of Mathematics, California State University Channel Islands, Camarillo, CA, USA}
\email{ivona.grzegorczyk@csuci.edu, ricardo.suarez532@csuci.edu}

\begin{document}

\normalsize

\begin{abstract} 
 In this paper we  construct complex tori, denoted by  $S_{\mathbb{B}_{1,p,q}}$, as quotients of tensor products of Cayley--Dickson algebras, denoted $\mathbb{B}_{1,p,q}=\C\otimes \mathbb{H}^{\otimes p}\otimes \mathbb{O}^{\otimes q}$, with their integral subrings. We then show that these complex tori have endomorphism rings of full rank 
and are isogenous to the direct sum of $2^{2p+3q}$ copies of an elliptic curve $E$ of $j$-invariant $1728$.
\end{abstract}

\maketitle

\section{Introduction }

In the  standard model of quarks and leptons (see \cite{Di1,Di2,Fu}), Cayley--Dickson algebras are used to construct an algebraic model of reality, where the complexified quaternions, complexfied octonions, and the tensor product $\C\otimes_{\R}\mathbb{H}\otimes_{\R} \mathbb{O}$ are all viewed as spaces of spinors.  Motivated by this algebraic model model of reality, we generalize this process to construct  complex tori with nontrivial endomorphism rings isogenous to the direct sum of certain elliptic curves, whose covering spaces are tensor products of Cayley--Dickson algebras. 
\bigskip

For any complex torus $T$ of dimension $n$, the endomorphism ring $\End_{\mathbb{Z}}(T)$  is a free $\mathbb{Z}$-module with the property that $\rk(\End_{\mathbb{Z}}(T))\leq 2\cdot n^2$. When the endomorphism ring is of full rank, we have the following proposition (see \cite{Sh}).

\begin{prop}\label{Decomposition}
Let $T$ be a complex torus of dimension $n$. If the rank of the endomorphism ring is $2 n^2$, then $T$ is isogenous to the direct sum of $n$ copies of an elliptic curve $E$ with complex multiplication.
\end{prop}

In this paper, using the work of Dixon (see  \cite{Di1}) and Furey (see \cite{Fu}), we construct complex tori as in the above proposition with nontrivial endomorphisms  from tensor products of Cayley--Dickson algebras (i.e.\ algebras obtained by the Cayley–-Dickson construction that produces a sequence of algebras over the field of real numbers, each with twice the dimension of the previous one, with suitably modified operations). The simplest examples are complex numbers, quaternions, and octonions, which are useful composition algebras frequently applied in mathematical physics.

We denote our constructed complex torus by $S_{\mathbb{B}_{1,p,q}}$, where $\mathbb{B}_{1,p,q}$ denotes the $\R$ tensor product $\C\otimes \mathbb{H}^{\otimes p}\otimes \mathbb{O}^{\otimes q}$. Here we interpret $\mathbb{B}_{1,p,q}$ as a spinor space for the adjoint algebra of componentwise left and right actions  $\mathbb{B}^{A}_{1,p,q}=\C\otimes_{\R} \mathbb{H}_{A}^{\otimes p}\otimes_{\R} \mathbb{O}_{A}^{\otimes q}$. Note that for any choice of positive integers $(p,q)$,  the  torus $S_{\mathbb{B}_{1,p,q}}$ is of dimension $2^{2p+3q}$. Moreover, we claim that 
$S_{\mathbb{B}_{1,p,q}}$ is a $(\mathbb{B}_{1,p,q}^{A})_{\mathbb{Z}}$-module, giving us the faithful $\mathbb{Z}$-module representations $\rho:(\mathbb{B}_{1,p,q}^{A})_{\mathbb{Z}}\rightarrow \End_{\mathbb{Z}}(S_{\mathbb{B}_{1,p,q}})$.  Hence we are able to conclude that the endomorphism ring of our complex torus is of full rank. Further, we show that $S_{\mathbb{B}_{1,p,q}}$ is a $\mathbb{Z}$-module and is isogenous to the direct sum of $2^{2p+3q}$ copies of an elliptic curve $E$ of $j$-invariant $1728$. Our main result about the construction of the tori is summarized in Proposition \ref{main result}.

\section{Tensor products of Cayley--Dickson algebras and induced complex tori}

This section is inspired by the description  of standard model physics as in  \cite{Di1,Di2,Fu} and the matrix algebra  isomorphisms found in \cite{Po}. As a consequence of Hurwitz's theorem \cite{Hu}, we know that the only  normed division algebras over $\R$ are the Cayley--Dickson algebras  $\R, \C, \mathbb{H}$, and $\mathbb{O}$ (i.e.\ the real and complex numbers, the quaternions, and the octonions). 
These algebras are of dimension  $1,2,4,$ and $8$, respectively, over the real numbers. 
\bigskip

\noindent{\bf Remark.} We denote the tensor product of division algebras of the form $\C\otimes \mathbb{H}^{\otimes p}\otimes \mathbb{O}^{\otimes q}$ by $\mathbb{B}_{1,p,q}$ (where  $\otimes$ denotes $\otimes_{\R}$). As a complex vector space, $\mathbb{B}_{1,p,q}$ has complex dimension $4^{p}\cdot 8^{q}=2^{2p+3q}$.  Additionally, we can view $\mathbb{B}_{1,p,q}$ as the complexification of the real form $\mathbb{B}^{\R}_{1,p,q}=\mathbb{H}^{\otimes p}\otimes_{\R} \mathbb{O}^{\otimes q}$ of real dimension $2^{2p+3q}$. 
\bigskip

In order to construct complex tori out of the above tensor products, we need to consider certain of their subrings, defined below. 

\begin{defn}\label{def: integral subring}
For any tensor product of division algebras $\mathbb{B}_{1,p,q}$, we denote by $\mathbb{B}_{1,p.q}^{\mathbb{Z}}$ 
the integral subring of $\mathbb{B}_{1,p,q}$ obtained by restricting the complex scalars to integral scalars in  $\mathbb{Z}$. 
    
\end{defn}

This integral subring is clearly a free $\mathbb{Z}$-submodule of $\mathbb{B}_{1,p,q}$ under addition; and when we view it as a real algebra, we can view $\mathbb{B}^{\mathbb{Z}}_{1,p,q}$ as $\mathbb{B}^{\mathbb{Z}}_{1,p,q}=\mathbb{B}_{0,p,q}^{\mathbb{Z}}\oplus  i \cdot \mathbb{B}_{0,p,q}^{\mathbb{Z}}$, where $\mathbb{B}_{0,p,q}^{\mathbb{Z}}$ is the integral subring of the real form $\mathbb{B}_{1,p,q}^{\R}$. Note that the integral subring $\mathbb{B}^{\mathbb{Z}}_{1,p,q}$ can also be viewed as the restriction of the ring $\C$ on the first tensor components to the Gaussian integers $\mathbb{Z}[i]$, together with the restriction of every copy of the quaternions and octonions to their integral subrings, denoted $\mathbb{H}_{\mathbb{Z}}$ and $\mathbb{O}_{\mathbb{Z}}$, where all copies are tensors over $\mathbb{Z}$ (that is, we go from $\otimes_{\R}$ to $\otimes_{\mathbb{Z}}$). This allows us to view $\mathbb{B}^{\mathbb{Z}}_{1,p,q}$ as $\mathbb{Z}[i]\otimes_{\mathbb{Z}} \mathbb{H}^{\otimes p}_{\mathbb{Z}}\otimes_{\mathbb{Z}} \mathbb{O}^{\otimes q}_{\mathbb{Z}}$ , and the integral subring of the real form as $\mathbb{B}_{0,p,q}^{\mathbb{Z}}=\mathbb{H}_{\mathbb{Z}}^{\otimes p}\otimes_{\mathbb{Z}} \mathbb{O}_{\mathbb{Z}}^{\otimes q}$. Note that the rank of this $\mathbb{Z}$-module is $2\cdot 4^{p}\cdot 8^{q}=2^{2p+3q+1}$; i.e., it has that many generators. Hence, we can conclude that the integral subring $\mathbb{B}_{1,p,q}^{\mathbb{Z}}$ is a full rank lattice of the complex vector space $\mathbb{B}_{1,p,q}$.  The above can be summarized in the following lemma.

\begin{lem}\label{quotient is a complex torus}
The quotient $S_{\mathbb{B}_{1,p,q}}:=\dfrac{\mathbb{B}_{1,p,q}}{\mathbb{B}_{1,p,q}^{\mathbb{Z}}}$  is a complex torus associated to the tensor product $\mathbb{B}_{1,p,q}$ of complex dimension $2^{2p+3q}$.
\end{lem}

Before describing natural actions on our complex torus  $S_{\mathbb{B}_{1,p,q}}$, we  define the adjoint algebra of actions on $\mathbb{B}_{1,p,q}$ as follows (see also \cite{Su}).

\begin{defn}\label{def: adjoint algebra}
The \textbf{adjoint algebra of left and right multiplication maps} is defined as  $\mathbb{B}_{1,p,q}^{A}=\{A_{x,y}:x,y\in \mathbb{B}_{1,p,q}\}$,  where $A_{x,y}(a)=R_x\circ L_y(a)=(y\cdot a)\cdot x$ for any $a,x,y\in\mathbb{B}_{1,p,q}$.
  
\end{defn}

Consider the  adjoint algebra  as a tensor algebra defined by  $\mathbb{B}_{1,p,q}^{A}=\mathbb{C}_{A}\otimes \mathbb{H}_{A}^{\otimes p}\otimes \mathbb{O}_{A}^{\otimes q}$, where $\C_{A},\ \mathbb{H}_A,$ and $\mathbb{O}_A$ are the adjoint algebras of left and right actions on $A = \C$, $\mathbb{H}$, or $\mathbb{O}$ respectively. The left and right  actions of  $\mathbb{B}_{1,p,q}^{A}$ on  $\mathbb{B}_{1,p,q}$ are considered componentwise actions on each tensor component by the appropriate adjoint algebra.  
 Now, we have the following isomorphisms for the above adjoint algebras of actions on the complex numbers, quaternions, and octonions described in  \cite{Di2,Fu,LM,Po}: 

\begin{enumerate}
    \item $\C_{A}\cong \C$.
      \item $\mathbb{H}_{A}\cong \mathbb{H}\otimes \mathbb{H}\cong \R(4)$.
    \item $ \mathbb{O}_{A}\cong \R(8)$.
    \item $\mathbb{C}\otimes \mathbb{H}\cong \C(2)$.
\end{enumerate}

\noindent Thus, $\C_{A},\mathbb{H}_{A},$ and $\mathbb{O}_{A}$ are matrix algebras, which makes $\mathbb{B}^{A}_{1,p,q}$ itself into  a matrix algebra described by the following isomorphisms of $\R$ algebras: 
\begin{align*}
\mathbb{B}^{A}_{1,p,q} & = \C\otimes \mathbb{H}_{A}^{\otimes p}\otimes \mathbb{O}_A^{\otimes q} \cong \C\otimes (\mathbb{H}\otimes \mathbb{H})^{p}\otimes \mathbb{O}_A^{\otimes q} \\
& \cong \C\otimes (\R(4))^{\otimes p}\otimes \R(8)^{\otimes q} \cong \C\otimes \R(4^{p})\otimes \R(8^{q}) 
 \cong \C\otimes \R(2^{2p+3q}) \\ & \cong \C(2^{2p+3q}). 
\end{align*}
Next, we can describe actions on our complex torus  $S_{\mathbb{B}_{1,p,q}}$  by restricting the adjoint algebra   $\mathbb{B}_{1,p,q}^{A}$ to the free $\mathbb{Z}$-module $(\mathbb{B}_{1,p,q}^{A})_{\mathbb{Z}}:=\mathbb{Z}[i]\otimes_{\mathbb{Z}} (\mathbb{H}_{\mathbb{Z}}^{A})^{\otimes p}\otimes_{\mathbb{Z}} (\mathbb{O}_{\mathbb{Z}}^{A})^{\otimes q}$. It follows immediately that the adjoint actions in $(\mathbb{B}_{1,p,q}^{A})_{\mathbb{Z}}$ preserve the full rank lattice $\mathbb{B}_{1,p,q}^{\mathbb{Z}}$. These actions extend to the complex torus $S_{\mathbb{B}_{1,p,q}}$, where we have $x_{A}\cdot [y]=[x_{A}\cdot y]$ for any $x_{A}\in (\mathbb{B}_{1,p,q}^{A})_{\mathbb{Z}}$ and $[y]\in S_{\mathbb{B}_{1,p,q}}$. 
Hence we obtain the following lemma.

\begin{lem}
The complex torus $S_{\mathbb{B}_{1,p,q}}$  has nontrivial actions given by the $\mathbb{Z}$-module $(\mathbb{B}^{A}_{1,p,q})_{\mathbb{Z}}$.
\end{lem}

 Now we state our main proposition about the complex tori described in Lemma \ref{quotient is a complex torus}.

\begin{prop}\label{main result}
Fix $(p,q)\in\mathbb{N}^2$. For the complex torus $S_{\mathbb{B}_{1,p,q}}$ we have a faithful \ $\mathbb{Z}$-module representation $\rho:(\mathbb{B}^{A}_{1,p,q})_{\mathbb{Z}}\rightarrow \End_{\mathbb{Z}}(S_{\mathbb{B}_{1,p,q}})$, given by $\rho(x_{A})[y]=[x_{A}\cdot y]$. Hence  our complex torus $S_{\mathbb{B}_{1,p,q}}$ is a $(\mathbb{B}^{A}_{1,p,q})_{\mathbb{Z}}$-module. Moreover,  $S_{\mathbb{B}_{1,p,q}}$  is a complex torus of dimension $2^{2p+3q}$ isogenous to the direct sum of $2^{2p+3q}$ copies of an elliptic curve $E$ of $j$-invariant $1728$.

\end{prop}

\begin{proof}
For the following isomorphism, we denote the tensor products $\otimes_{\mathbb{Z}}$ as $\otimes$.
\begin{align*}
(\mathbb{B}^{A}_{1,p,q})_{\mathbb{Z}}& =\mathbb{Z}[i]\otimes (\mathbb{H}^{\mathbb{Z}}_{A})^{\otimes p}\otimes (\mathbb{O}^{\mathbb{Z}}_A)^{\otimes q}\cong \mathbb{Z}[i]\otimes  (\mathbb{H}_{\mathbb{Z}}\otimes \mathbb{H}_{\mathbb{Z}})^{\otimes p}\otimes (\mathbb{O}^{\mathbb{Z}}_A)^{\otimes q} \\ & \cong \mathbb{Z}[i]\otimes (\mathbb{Z}(4))^{\otimes p}\otimes \mathbb{Z}(8)^{\otimes q}
\cong\mathbb{Z}[i]\otimes \mathbb{Z}(4^{p})\otimes \mathbb{Z}(8^{q})\cong \mathbb{Z}[i]\otimes \mathbb{Z}(2^{2p+3q}) \\ & \cong\mathbb{Z}(2^{2p+3q})^{\oplus 2}.
\end{align*}
Hence it is immediate that  the rank, as a $\mathbb{Z}$-module, of $(\mathbb{B}^{A}_{1,p,q})_{\mathbb{Z}}$ is $2^{4p+6q+1}$.

We now connect $(\mathbb{B}^{A}_{1,p,q})_{\mathbb{Z}}$ with the endomorphisms of our complex torus via 
the map $\rho:(\mathbb{B}_{1,p,q})_{\mathbb{Z}}\rightarrow \End_{\mathbb{Z}}(S_{\mathbb{B}_{1,p,q}})$, given by $\rho(x_{A})[y]=[x_{A}\cdot y]$, for any adjoint action $x_{A}\in (\mathbb{B}^{A}_{1,p,q})_{\mathbb{Z}}$. This map is clearly an injection, since its kernel is trivial. The triviality of the kernel is immediate since for any generator  $x_{A}\in\mathbb{Z}[i]\otimes_{\mathbb{Z}}\mathbb{H}_{\mathbb{Z}}^{A}\otimes_{\mathbb{Z}} \mathbb{O}_{\mathbb{Z}}^{A}$ chosen as an adjoint action, only the zero element annihilates all elements in the tensor algebra, since $(\mathbb{B}^{A}_{1,p,q})_{\mathbb{Z}}$ is the $\mathbb{Z}$ tensor product of the integral subrings of associative division algebras, where each associative algebra does not carry any zero divisors.  Thus $\rho$ is an injective module homomorphism, i.e.\ a faithful representation of  $\mathbb{Z}$-modules. Given that $\rho:(\mathbb{B}_{1,p,q}^{A})_{\mathbb{Z}}\rightarrow \End_{\mathbb{Z}}(S_{\mathbb{B}_{1,p,q}})$ is a faithful representation of the $\mathbb{Z}$-module $(\mathbb{B}_{1,p,q}^{A})_{\mathbb{Z}}$, we immediately have that $\rho((\mathbb{B}_{1,p,q}^{A})_{\mathbb{Z}})$ is a $\mathbb{Z}$-submodule of $\End_{\mathbb{Z}}(S_{\mathbb{B}_{1,p,q}})$ of rank $2^{4p+6q+1}$. 
Given that $S_{\mathbb{B}_{1,p,q}}$ is a complex torus, using Proposition \ref{Decomposition}, we can estimate the rank of the endomorphism ring as follows:
\[
\rk(\End_{\mathbb{Z}}(S_{\mathbb{B}_{1,p,q}}))\leq 2\cdot (\dim_{\C} S_{\mathbb{B}_{1,p,q}})^2=2\cdot (2^{2p+3q})^2=2^{4p+6q+1}. 
\]
In our case, since $\End_{\mathbb{Z}}(S_{\mathbb{B}_{1,p,q}})$ contains a $\mathbb{Z}$-submodule of rank $2^{4p+6q+1}$, via our faithful representation, we can conclude that the rank of our endomorphism ring is equal to $2^{4p+6q+1}$. Hence, we have an injective faithful representation of $\mathbb{Z}$-modules of equal rank. Therefore, from Proposition \ref{Decomposition}, we know that $S_{\mathbb{B}_{1,p,q}}$ is isogenous to the direct sum of $2^{2p+3q}$ copies of an elliptic curve with complex multiplication; i.e., $S_{\mathbb{B}_{1p,q}}\sim \underbrace{E\oplus \cdots \oplus E}_{2^{2p+3q} \textrm{ times}}$, where $E$ is an elliptic curve with complex multiplication.   The endomorphism ring $\End_{\mathbb{Z}}(S_{\mathbb{B}_{1,p,q}})$ carries endomorphisms of order $1,2$, or $4$ injected by the lattice generators 
 of $(\mathbb{B}^{A}_{1,p,q})_{\mathbb{Z}}$. Hence, extending these actions through the isogeny, it is immediate that $E$ is an elliptic curve that itself carries automorphisms of order $4$; and thus it is an elliptic curve of $j$-invariant $1728$ (see  \cite{Su}).  
\end{proof}

Therefore, we can conclude that for any given choice of $(p,q)$, we have a complex torus, constructed from the tensor product of Cayley--Dickson algebras, isogenous to the direct sum of $2^{2p+3q}$ copies of an elliptic curve $E$ of $j$-invariant $1728$.  We now extend our analysis to define  the analytic representations of our complex tori $S_{\mathbb{B}_{1,p,q}}$. For the endomorphism algebra of the real form  $\mathbb{B}^{\R}_{1,p,q}=\mathbb{B}_{0,p,q}$, we have the following $\R$-endomorphism algebra isomorphisms (see \cite{Ha,LM,Po}):  
\[
\End_{\R}(\mathbb{B}^{\R}_{1,p,q})\cong \End_{\mathbb{R}}(\mathbb{H})^{\otimes p}\otimes_{\mathbb{R}} \End_{\mathbb{R}}(\mathbb{O})^{\otimes q}\cong \R(4)^{\otimes p}\otimes_{\mathbb{R}} \R(8)^{\otimes q}\cong \R(2^{2p+3q}).
\]
Now, by the transitivity of $\R$ algebra isomorphisms, we identify the following real forms isomorphically: $\End_{\R}(\mathbb{B}^{\R}_{1,p,q})\cong \R(2^{2p+3q})\cong (\mathbb{B}_{1,p,q}^{A})^{\R}$. Hence one can easily see that $\End_{\R}(\mathbb{B}_{1,p,q}^{\R})$ is $\R$ algebra of dimension $2^{4p+6q}$. When we complexify the real form, we get a complex matrix algebra of dimension $2^{4p+6q}$; that is, 
\[
\End_{\R}(\mathbb{B}^{\mathbb{R}}_{1,p,q})\otimes_{\R} \C\cong \R(2^{2p+3q})\otimes_{\R} \C\cong \C(2^{2p+3q})\cong \mathbb{B}^{A}_{1,p,q}.
\]
Therefore, we can establish an isomorphism between the complex endomorphisms of the tangent space at the origin of our complex torus, $T_{0}S_{\mathbb{B}_{1,p,q}}=\mathbb{B}_{1,p,q}$,  and the adjoint algebra of actions on the space $\mathbb{B}_{1,p,q}$. Denoting this isomorphism by $f:\mathbb{B}^{A}_{1,p,q}\xrightarrow{\cong} \End_{\C}(\mathbb{B}_{1,p,q})$, we can  then define the \textit{analytic representation} of our complex tori, $\tau_{a}:\End_{\mathbb{Z}}(S_{\mathbb{B}_{1,p,q}})\rightarrow \End_{\mathbb{C}}(\mathbb{B}_{1,p,q})$, via $\tau_{a}(\theta)=f(\rho^{-1}(\theta))$, where we can view $f(\rho^{-1}(\theta))$ as a $2^{2p+3q}\times 2^{2p+3q}$  complex matrix representation.


\begin{thebibliography}{EMG}

\bibitem{BL} C.~Birkenhake and H.~Lange, \emph{Complex Abelian Varieties}, second edition. Grundlehren der Mathematischen Wissenchaften Vol 302, Springer-Verlag, Berlin, 2004.


\bibitem{Br} R.~Brown,  \emph{On Generalized Cayley-Dikson Algebras}. Pacific Journal Math. 20(3), 415-422, 1967. 

\bibitem{Di1} G.~Dixon, \emph{Derivation of the standard model}, Nuovo Cimento B 105(3), 349-364, 1990.

\bibitem{Di2} G.~M.~Dixon, \emph{Division Algebras: Octonions, Quaternions, Complex numbers and the Algebraic Design of Physics}. Springer-Science+ Business Media B.V. 290, 1994.

\bibitem{Fu} C.~Furey,  \emph{Standard model physics from an algebra?}. PhD thesis, University of Waterloo, Ontario, Canada, 2015.

\bibitem{Ha} F.~Harvey,  \emph{Spinors and Calibrations}. Academic Press Inc, San Diego, 1990.

\bibitem{Hu} A.~Hurwitz,  \emph{Uber dic Komosition der Quadratischen Formen}. Math. Ann. 88(1-2):1-25. 

\bibitem{LM} H.~Lawson and M.-L.~Michelsohn, \emph{Spin Geometry}. Princeton University Press, Princeton, 1989.

\bibitem{Po} I.~R.~Porteous, \emph{Clifford Algebras and the Classical Groups}. Cambridge University Press, Cambridge, 1995.

\bibitem{Sh} A Shimizu, \emph{On complex tori with many  endomorphisms}. Tsukuba Journal of Mathematics 8, no. 2  1984: 297–318. http://www.jstor.org/stable/43685465.

\bibitem{Su} R.~Su\'{a}rez, Clifford Multiplication on spinor Abelian varieties and algebraic curves. Dissertation, Universit\`{a} degli Studi di Torino, 2024.
\end{thebibliography}
\end{document}